\documentclass[12pt]{article}
\usepackage{mathrsfs}
\usepackage{amsfonts}

\usepackage{amssymb,amsthm,amsmath,hyperref}

 \numberwithin{equation}{section}

\newtheorem{thm}{Theorem}[section]
\newtheorem{lem}[thm]{Lemma}
\newtheorem{prop}[thm]{Proposition}
\newtheorem{cor}[thm]{Corollary}

\theoremstyle{definition}
\newtheorem{rem}[thm]{Remark}

\allowdisplaybreaks

\begin{document}

\title{A Blow-up Criterion for Two Dimensional
Compressible Viscous Heat-Conductive Flows\thanks{Supported by NSFC
10871175, 10931007, 10901137, Zhejiang Provincial Natural Science
Foundation of China (Z6100217), and SRFDP 20090101120005.} }
\author{Daoyuan Fang\thanks{E-mail: dyf@zju.edu.cn }, {Ruizhao Zi
\thanks{E-mail: ruizhao3805@163.com} and {Ting Zhang\thanks{E-mail: zhangting79@zju.edu.cn } }
}\\
\textit{\small Department of Mathematics, Zhejiang University,
Hangzhou 310027, China} }
\date{}
\maketitle

\begin{abstract}
We establish a blow-up criterion in terms of the upper bound of the
density and temperature for the strong solution to 2D compressible
viscous heat-conductive flows. The initial vacuum is allowed.
\end{abstract}
\section{Introduction}
This paper is concerned with a blow-up criterion for the two
dimensional compressible viscous heat-conductive flows in $(0,
T)\times\Omega$
\begin{eqnarray}\label{1.1}
\left\{ \begin{array}{ll}
\rho_t+\textrm{div}(\rho u)=0,\\[3mm]
(\rho u)_t+\textrm{div}(\rho u\otimes u)+\nabla
P(\rho,\theta)=\mu\Delta u+(\lambda+\mu)\nabla \textrm{div} u,\\[3mm]
c_{\nu}((\rho\theta)_t+\textrm{div}(\rho\theta
u))-\kappa\Delta\theta+P(\rho,\theta)\textrm{div}u=\frac{\mu}{2}|\nabla
u+\nabla u^{T}|^2+\lambda(\textrm{div} u)^2,
 \end{array}
 \right.
\end{eqnarray}
together with the initial-boundary conditions
\begin{eqnarray}
\label{1.2}&&(\rho, u, \theta)=(\rho_0, u_0, \theta_0)\ \ \textrm{in}\ \ \Omega,\\
\label{1.3}&&(u, \theta)=(0,0)\ \ \textrm{on}\ \
(0,T)\times\partial\Omega.
\end{eqnarray}
Here $\Omega$ is a bounded smooth domain in $\mathbb{R}^2$, and
$(\rho, u, \theta)$ are the density, velocity and temperature of the
fluid, respectively. For a perfect gas, the pressure is given by
\begin{equation}\label{1.5}
P(\rho,\theta)=R\rho\theta,
\end{equation}
where $R>0$ is a generic gas constant. The viscous coefficients
$\mu$ and $\lambda$ are constants satisfying
\begin{equation}\label{1.6}
\mu>0, \quad \lambda+\mu\geq0.
\end{equation}
Finally, $c_{\nu}>0$ and $\kappa>0$ are the specific heat at a
constant volume  and thermal conductivity coefficient, respectively.

Some of the previous works in this direction can be summarized as
follows. In the absence of vacuum, Matsumura and Nishida
\cite{Matsumura} proved the global well-posedness of the classical
solution to (\ref{1.1})-(\ref{1.6}) with the initial data close to
an equilibrium state and Danchin \cite{Danchin1,Danchin2} considered
similar problems in the framework of critical spaces. With the aid
of the effective viscous flux $F$, Hoff \cite{Hoff2} proved the
global existence of weak solutions with less restrictions on the
initial data. If vacuum is taken into account, the problem becomes
more complicated. Recently, Wen and Zhu \cite{Wen} obtained a unique
global classical solution to the 1D model with large initial data
and vacuum. For spherically symmetric flow in the exterior domain
without the origin, the global existence of a strong solution was
obtained by Jiang \cite{Jiang}. To our best knowledge, the first
attempt towards the existence of weak solutions to the full
compressible Navier-Stokes equations in dimension $N \geq2$ is given
by Feireisl \cite{Feireisl2} where he proved the global existence of
the so-called variational solutions in the case of real gases. In
addition, it is worth mentioning that, by using a new mathematical
entropy equality, Bresch and Desjardins \cite{Bresch} got the global
weak solutions to the Navier-Stokes equations for heat conducting
fluids with density and temperature dependent viscosity. Certainly,
for the isentropic case, the results are more satisfactory, see
\cite{Feireisl1,Hoff1,Huang4,Jiang2,Jiang3,Lions} and references
therein.

On the other hand, when the initial density is compactly supported,
Xin \cite{Xin} proved that a smooth solution will blow-up in finite
time in the whole space, see also \cite{Rozanova} for a more general
blow-up result. Thus, it is interesting to investigate the mechanism
of blow-up and the structure of possible singularities. Some
progress for the isentropic flow can be found in
\cite{Huang1,Huang3,Huang5,Sun1,Sun2} and references therein. For
the non-isentropic case, Fan and Jiang \cite{Fan-Jiang1} proved the
following blow-up criteria for the local strong solutions to
(\ref{1.1})-(\ref{1.6}) in the case of two dimensions
\begin{equation*}
\lim_{T\rightarrow T^*} \bigg( \sup_{0\leq t\leq T}
\{\|\rho\|_{L^\infty}, \|\rho^{-1}\|_{L^\infty},
\|\theta\|_{L^\infty}\}+\int_0^T(\|\rho\|_{W^{1,q_0}}+\|\nabla\rho\|_{L^2}^4+
\|u\|_{L^{r,\infty}}^{\frac{2r}{r-2}})dt\bigg)=\infty,
\end{equation*}
or
\begin{equation*}
\lim_{T\rightarrow T^*} \bigg( \sup_{0\leq t\leq T}
\{\|\rho\|_{L^\infty}, \|\rho^{-1}\|_{L^\infty},
\|\theta\|_{L^\infty}\}+\int_0^T(\|\rho\|_{W^{1,q_0}}+\|\nabla\rho\|_{L^2}^4)dt\bigg)=\infty,
\end{equation*}
provided $2\mu>\lambda$, where $T^*<\infty$ is the maximal time of
existence of a strong solution, $q_0>3$ is a certain number, $3 < r
\leq\infty$ with $2/s+3/r = 1$, and $L^{r,\infty}\equiv
L^{r,\infty}(\Omega)$ is the Lorentz space. If the domain is a
periodic or unit square domain in $\mathbb{R}^2$, the blow-up
criterion is refined by Jiang and Ou \cite{Jiang1} to be
\begin{equation}\label{1.7}
\lim_{T\rightarrow T^*} \int_0^T\|\nabla u\|_{L^\infty}dt=\infty,
\end{equation}
which coincides with the Beale-Kato-Majda criterion \cite{B-K-M} for
ideal incompressible flows. For the case that $\Omega$ is a bounded
domain in $\mathbb{R}^3$, Fan, Jiang and Ou \cite{Fan-Jiang2}
established a blow-up criterion with additional upper bound of the
temperature $\theta$
\begin{equation}\label{1.8}
\lim_{T\rightarrow T^*} \bigg( \sup_{0\leq t\leq T}
\|\theta\|_{L^\infty}+\int_0^T\|\nabla
u\|_{L^\infty}dt\bigg)=\infty,
\end{equation}
provided
\begin{equation}\label{1.9}
7\mu>\lambda.
\end{equation}
Very recently, under the same condition (\ref{1.9}), Sun, Wang and
Zhang \cite{Sun3} obtained a blow-up criterion in terms of the upper
bound of $(\rho,\rho^{-1},\theta)$
\begin{equation}\label{1.10}
\lim_{T\rightarrow T^*} \bigg( \sup_{0\leq t\leq T}
\{\|\rho\|_{L^\infty}+\|\rho^{-1}\|_{L^\infty}+\|\theta\|_{L^\infty}\}\bigg)=\infty.
\end{equation}

However, from the physical point of view, it is natural to expect
that the solution does not blow-up under conditions on the
non-appearance of the concentration of the temperature and the
density. For this reason, the aim of the current paper is to remove
the lower bound of the density $\rho$ in (\ref{1.10}). As pointed
out in \cite{Sun3}, without the lower bound of the density $\rho$,
it is difficult to deal with the highly nonlinear terms $|\nabla
u+\nabla u^{\top}|^2, (\textrm{div}u)^2$ in the temperature
equation. To overcome this difficulty, we put together the estimates
of $\sup_{0\leq t\leq T}\|\nabla u\|_{L^2}^2$, $\sup_{0\leq t\leq
T}\|\sqrt{\rho} \dot{u}\|_{L^2}^2+\int_0^T\|\nabla\dot{u}\|_{L^2}^2$
and $\sup_{0\leq t\leq T}\|\nabla
\theta\|_{L^2}^2+\int_0^T\|\rho\dot{\theta}\|_{L^2}^2$ (see Lemma
\ref{lem3.1}-\ref{lem3.3} below) and make full use the good term
$\int_0^T\|\rho\dot{\theta}\|_{L^2}^2$. This is the main ingredient
of our proof. Besides, let us emphasize that, instead of
$\dot{\theta}$, we use $\theta_t$ as the test function in the proof
of Lemma \ref{lem3.3}, which enable us to rewrite
$$
\int_0^t[\frac{\mu}{2}|\nabla u+\nabla
u^{\top}|^2+\lambda(\textrm{div}u)^2]\theta_t
$$
as
$$\frac{d}{dt}\int_0^t[\frac{\mu}{2}|\nabla u+\nabla
u^{\top}|^2+\lambda(\textrm{div}u)^2]\theta+\textrm{other terms},
$$
and thus we can avoid using the lower bound of the density. Finally,
we remark that in the process of our proof, we have to deal with the
troublesome term $\int_0^T\|\nabla u\|_{L^4}^4$. To this end, we use
the decomposition of the velocity $u=v+w$ introduced by Sun, Wang
and Zhang in \cite{Sun1}. More precisely, let $v$ solves the
elliptic system
\begin{equation}\label{4.1}
\begin{cases}
\mu\Delta v+(\lambda+\mu)\nabla\textrm{div}v=\nabla P(\rho,\theta)\
\ \textrm{in}\ \  \Omega,\\
v=0\ \ \textrm{on}\ \ \partial\Omega.
\end{cases}
\end{equation}
then from the momentum equation $(\ref{1.1})_2$ and (\ref{4.1}), it
is easy to see that $w$ solves
\begin{eqnarray}\label{4.2}
\begin{cases}
\mu\Delta w+(\lambda+\mu)\nabla\textrm{div}w=\rho\dot{u}\ \
\textrm{in}\ \ \Omega,\\
w=0\ \ \textrm{on}\ \ \partial\Omega.
\end{cases}
\end{eqnarray}
Here $w$ is an important quantity, whose divergence could be
regarded as the effective viscous flux and thus possesses more
regularity information than $u$. Based on this decomposition, our
approach next is to use the Gagliardo-Nirenberg inequality and then
close the estimates with the help of Gronwall's inequality.
Unfortunately, if the dimension $N=3$, this approach does not work
anymore, and that is why we only consider the two dimensional case.
\bigbreak

We would like to give some notations which will be used throughout
the paper.

\noindent\textbf{Notations:}
\begin{enumerate}
  \item Throughout this paper, we denote by $\Omega$ a bounded smooth domain in $\mathbb{R}^2$ and $L^r=L^r(\Omega),\ \ W^{k,r}=W^{k,r}(\Omega),\ \ H^k=W^{k,2},$\\[1mm]
  $H_0^1=\{v\in H^1(\Omega):\|v\|_{H^{1}}<\infty\
  \textrm{and}\ v|_{\partial\Omega}=0 \   \textrm{in the sence of trace}\}$,
  by virtue of Poincar\'{e}'s inequality, we redefine
  $\|v\|_{H_0^1}=\|\nabla v\|_{L^2}$.
  \item $\dot{f}=f_t+u\cdot\nabla f$ is the material derivative
  of $f$, and $\ddot{f}=(\dot{f})\dot{}$ denotes the twice material derivative
  of $f$.
  \item $\displaystyle\int f=\int_{\Omega}f dx$ and $\displaystyle\int_0^T\int f=\int_0^T\int_{\Omega}f dxdt.$
\end{enumerate}

Before proceeding any further, we recall that Cho and Kim
\cite{Cho-Kim} obtained the local strong solution to
(\ref{1.1})-(\ref{1.6}) with initial vacuum for the spacial
dimension $N=3$ (and $\Omega$ need not be bounded). Now if $\Omega$
is a bounded smooth domain in $\mathbb{R}^2$ with $(u,\theta)$
satisfying the Dirichlet boundary condition (\ref{1.3}), it is not
difficult to verify that Cho and Kim's proof in \cite{Cho-Kim} still
works, and we state the corresponding result below.
\begin{prop}\label{prop1}
Let $q\in(2,\infty)$ be a fixed constant. Assume that the initial
data satisfy
\begin{eqnarray}\label{pro1}
\nonumber&\rho_0\geq0,\quad \rho_0\in W^{1, q},\ \
 (u_0,\theta_0)\in H_0^1\cap H^2, &
\end{eqnarray}
and the compatibility conditions
\begin{eqnarray}\label{pro2}
\nonumber&\mu\Delta u_0+(\lambda+\mu)\nabla
\textrm{div}u_0-R\nabla(\rho_0\theta_0)=\rho_0^{\frac{1}{2}}g_1,&\\
&\kappa\Delta\theta_0+\frac{\mu}{2}|\nabla u_0+\nabla
u^{\top}_0|^2+\lambda(\textrm{div}u_0)^2=\rho_0^{\frac{1}{2}}g_2,
\end{eqnarray}
for some $(g_1,g_2)\in L^2$. Then there exist a positive constant
$T_0$ and a unique strong solution $(\rho, u, \theta)$ to
(\ref{1.1})-(\ref{1.6}) such that
\begin{eqnarray}\label{pro3}
\nonumber&\rho\geq0,\quad \rho\in C([0,T_0];
W^{1, q}),&\\
\nonumber& (u, \theta)\in C([0,T_0];H_0^1\cap
H^2)\cap L^2(0,T_0;W^{2,q}),&\\
&(u_t,\theta_t)\in L^2([0, T_0];H_0^1),\quad
(\sqrt{\rho}u_t,\sqrt{\rho}\theta_t)\in L^\infty([0,T_0];L^2).&
\end{eqnarray}
\end{prop}

Our main result is stated as follows.
\begin{thm}[\textbf{Blow-up Criterion}]\label{thmBC}
Suppose that the assumptions in Proposition \ref{prop1} are
satisfied, and assume that $(\rho, u, \theta)$ is the strong
solution constructed in Proposition \ref{prop1}. Let $T^*$ be the
maximal existence time. If $T^*<\infty$, then
\begin{equation}\label{bc}
\lim_{T\rightarrow T^*}\sup_{0\leq t\leq
T}(\|\rho\|_{L^\infty}+\|\theta\|_{L^\infty})=\infty.
\end{equation}
\end{thm}
\begin{rem}
If either the boundary condition (\ref{1.3}) on $\theta$ is replaced
by the Neumann boundary condition $\frac{\partial\theta}{\partial
n}|_{\partial\Omega}=0$ or the bounded domain $\Omega$ is replaced
by the whole space $\mathbb{R}^2$, we are able to obtain the same
blow-up criterion as (\ref{bc}) provided the corresponding local
strong solution exists as in Proposition \ref{prop1}. In these
cases, we can not use Poincar\'{e}'s inequality directly to get the
lower order estimates of $u$ and $\theta$, however, these would be
recovered by using Lions's (for $\Omega$ bounded) and Hoff's (for
$\Omega=\mathbb{R}^2$) technics, and the readers are referred to
\cite{Lions} Remark 5.1 and \cite{Hoff3} Lemma 2.3, for instance.
\end{rem}
\section{Preliminaries}
Obviously, the appearance of the systems (\ref{4.1}) and (\ref{4.2})
make the estimates on the strongly elliptic operator
$\mu\Delta+(\lambda+\mu)\nabla\textrm{div}$ necessary. Now let $U$
solves
\begin{equation}\label{2.7}
\begin{cases}
\mu\Delta U+(\lambda+\mu)\nabla\textrm{div}U=F\ \ \textrm{in}\ \
\Omega,\\
U=0\ \ \textrm{on}\ \ \partial\Omega.
\end{cases}
\end{equation}
First of all, we state some classical estimates for the above
strongly elliptic systems, which will be used later frequently .
\begin{lem}[\cite{Sun1}]\label{lem2.0}Let $p\in(1,\infty)$ and $U$ be a solution
of (\ref{2.7}), then there exists a constant $C$ depending only on
$\lambda, \mu, p$ and $\Omega$ such that the following estimates
hold:

\noindent(1) if $F\in L^p$, then
\begin{equation}\label{2.01}
\|U\|_{W^{2,p}}\leq C\|F\|_{L^p},
\end{equation}

\noindent(2) if $F=\textrm{div} f$ with $f=(f_{i,j})_{2\times2},
f_{ij}\in L^p$, then
\begin{equation}\label{2.02}
\|U\|_{W^{1,p}}\leq C\|f\|_{L^p}.
\end{equation}
\end{lem}

As in \cite{Sun1,Sun2,Sun3}, we also need an endpoint estimate for
the strongly elliptic operator
$\mu\Delta+(\lambda+\mu)\nabla\textrm{div}$ for the case $p=\infty$.
This will be done with the aid of the John-Nirenberg space of
bounded mean oscillation whose norm is defined by
\begin{equation*}
\|f\|_{BMO(\Omega)}:=\|f\|_{L^2(\Omega)}+[f]_{BMO(\Omega)},
\end{equation*}
with
\begin{eqnarray*}
&&[f]_{BMO(\Omega)}:=\sup_{x\in \Omega,\
r\in(0,d)}\frac{1}{|\Omega_r(x)|}\int_{\Omega_r(x)}|f(y)-f_{\Omega_r(x)}|dy,\\
&&f_{\Omega_r(x)}=\frac{1}{|\Omega_r(x)|}\int_{\Omega_r(x)}f(y)dy,
\end{eqnarray*}
where $\Omega_r(x)=B_r(x)\cap\Omega$, $B_r(x)$ is a ball with center
$x$ and radius $r$, $d$ is the diameter of $\Omega$ and
$|\Omega_r(x)|$ denotes the Lebesgue measure of $\Omega_r(x)$. We
have
\begin{lem}[\cite{Acquistapace}]\label{lem2.3}
If $F=\textrm{div}f$ with $f=(f_{ij})_{2\times2}$, $f_{ij}\in
L^\infty\cap L^2$, then $\nabla U\in BMO(\Omega)$ and there exists a
constant $C$ depending only on $\lambda, \mu$ such that
\begin{equation}\label{2.8}
\|\nabla U\|_{BMO(\Omega)}\leq C(\|f\|_{L^\infty}+\|f\|_{L^2}).
\end{equation}
\end{lem}
Next, we state a variant of the Brezis-Waigner inequality
\cite{Wainger} which, together with Lemma \ref{lem2.3}, will be used
to give the gradient estimate of $\rho$.
\begin{lem}[\cite{Sun1}]\label{lem2.4}
Let $\Omega$ be a bounded Lipschitz domain in $\mathbb{R}^2$ and
$f\in W^{1, q}$ with $q\in(2,\infty)$, there exists a constant $C$
depending on $q$ such that
\begin{equation}\label{2.9}
\|f\|_{L^\infty}\leq C(1+\|f\|_{BMO}\ln(e+\| f\|_{W^{1, q}})).
\end{equation}
\end{lem}
\begin{rem}\label{rem2.5}
Lemma \ref{lem2.3} and \ref{lem2.4} are the two dimensional version
of the corresponding lemmas in \cite{Sun2}. To our knowledge, when
the domain $\Omega$ is bounded, the estimate (\ref{2.8}) can be
found in \cite{Acquistapace} for a more general setting, and for the
case $\Omega=\mathbb{R}^3$, Sun, Wang and Zhang  gave a proof in
\cite{Sun2}. The Brezis-Waigner type inequality(\ref{2.9}) was first
established in \cite{Kozono} on $\mathbb{R}^3$ and the case $\Omega$
be a $\textrm{bounded domain in}\ \mathbb{R}^N ( N=2,3)$ can be
found in \cite{Sun1,Sun2}.
\end{rem}

\section{Regularity of the velocity and temperature}
Let $0<T<T^*$ be arbitrary but fixed. In what follows, we assume
that $(\rho, u, \theta)$ is a strong solution of
(\ref{1.1})-(\ref{1.6}) on $[0,T]\times\Omega$ with the regularity
stated in Proposition \ref{prop1}. Suppose that $T^*<\infty$. We
will prove Theorem \ref{thmBC} by a contradiction argument. To this
end, we suppose that for any $T< T^*$,
\begin{equation}\label{3.00}
\sup_{0\leq t\leq T}(\|\rho\|_{L^\infty}+\|\theta\|_{L^\infty})\leq
M,
\end{equation}
where $M$ is independent of $T$. We will deduce a contradiction to
the maximality of $T^*$. Hereinafter, we denote by $C$ a general
positive constant which may depend on the initial data, the domain
$\Omega$, $M$ in (\ref{3.00}) and the maximal existence time $T^*$.
\begin{lem}\label{lem3.0}Under assumption (\ref{3.00}),  there holds for any $T<T^*$
\begin{equation}\label{3.0}
\sup_{0\leq t\leq
T}\int(\rho|u|^2+\rho|\theta|^2)+\int_0^T\int(|\nabla u|^2+|\nabla
\theta|^2)\leq C.
\end{equation}
\end{lem}
The proof is the same as Lemma 2 in \cite{Sun3}, we omit it here.

\begin{lem}\label{lem3.1}Under assumption (\ref{3.00}),  there holds for any $t<T^*$
\begin{equation}\label{3.1}
\int|\nabla u|^2+\int_0^t\int\rho|\dot{u}|^2\leq
C+C\int_0^t\int\sqrt{\rho}|\dot{\theta}||\nabla
u|+C\int_0^t\int|\nabla u|^4.
\end{equation}
\end{lem}
\begin{proof}
First of all, the momentum equation can be rewritten as
\begin{equation}\label{3.2}
\rho\dot{u}+\nabla P=\mu\Delta u+(\lambda+\mu)\nabla\textrm{div}u.
\end{equation}
Taking the $L^2$ inner product of the above equation with $\dot{u}$,
and integrating by parts, we get
\begin{equation}\label{3.3}
\int\rho|\dot{u}|^2=\int P\textrm{div}\dot{u}+\int[\mu\Delta
u+(\lambda+\mu)\nabla\textrm{div}u]\dot{u}.
\end{equation}
Direct calculation yields
\begin{eqnarray}\label{3.4}
\int P\textrm{div}\dot{u}&\nonumber=&\frac{d}{dt}\int P\textrm{div}u-\int(P_t+\textrm{div}(Pu))\textrm{div}u\\
&&+\int P\nabla u:\nabla u^{\top},
\end{eqnarray}
and
\begin{eqnarray}\label{3.5}
&\nonumber&\int[\mu\Delta
u+(\lambda+\mu)\nabla\textrm{div}u]\dot{u}\\
&\nonumber=&-\frac{1}{2}\frac{d}{dt}\int[\mu|\nabla
u|^2+(\lambda+\mu)(\textrm{div}u)^2]-\mu\int\nabla u:(\nabla u\nabla
u)\\
&&+\frac{\mu}{2}\int\textrm{div}u|\nabla
u|^2-(\lambda+\mu)\int\textrm{div}u\nabla u:\nabla
u^{\top}+\frac{\lambda+\mu}{2}\int(\textrm{div}u)^3.
\end{eqnarray}
Substituting (\ref{3.4}) and (\ref{3.5}) into (\ref{3.3}), we have
\begin{eqnarray}\label{3.6}
\nonumber&&\int\rho|\dot{u}|^2+\frac{1}{2}\frac{d}{dt}\int[\mu|\nabla
u|^2+(\lambda+\mu)(\textrm{div}u)^2]\\
\nonumber&=&\frac{d}{dt}\int
P\textrm{div}u-\int(P_t+\textrm{div}(Pu))\textrm{div}u+\int P\nabla
u:\nabla u^{\top}\\
\nonumber&&-\mu\int\nabla u:(\nabla u\nabla
u)+\frac{\mu}{2}\int\textrm{div}u|\nabla
u|^2\\
&&-(\lambda+\mu)\int\textrm{div}u\nabla u:\nabla u^{\top}
+\frac{\lambda+\mu}{2}\int(\textrm{div}u)^3.
\end{eqnarray}
Notice that
\begin{equation}\label{3.7}
P_t+\textrm{div}(Pu)=R\rho\dot{\theta},
\end{equation}
in view of the state equation $P(\rho,\theta)=R\rho\theta$ and
$(\ref{1.1})_1$. Substituting (\ref{3.7}) into (\ref{3.6}), and
integrating the resulting equation over $[0, t]$, we get
\begin{eqnarray}\label{3.8}
\nonumber&&\int_0^t\int\rho|\dot{u}|^2+\frac{1}{2}\int[\mu|\nabla
u|^2+(\lambda+\mu)(\textrm{div}u)^2]\\
\nonumber&\leq&C+\int
P\textrm{div}u+C\int_0^t\int\sqrt{\rho}|\dot{\theta}||\nabla
u|+C\int_0^t\int|\nabla u|^2+C\int_0^t\int|\nabla u|^3\\
&\leq&C_\epsilon+\epsilon\int|\nabla
u|^2+C\int_0^t\int\sqrt{\rho}|\dot{\theta}||\nabla
u|+C\int_0^t\int|\nabla u|^4,
\end{eqnarray}
where we have used Cauchy's inequality, (\ref{3.00}) and
(\ref{3.0}). After the second term on the right hand side of
(\ref{3.8}) absorbed  by the left hand side for a fixed $\epsilon$
small enough, we complete the proof of Lemma \ref{lem3.1}.
\end{proof}

\begin{lem}\label{lem3.2}Under assumption (\ref{3.00}),  there holds for any $t<T^*$
\begin{equation}\label{3.9}
\int\rho |\dot{u}|^2+\int_0^t\int|\nabla\dot{u}|^2\leq
C+C\int_0^t\int\rho|\dot{\theta}|^2+C\int_0^t\int|\nabla u|^4.
\end{equation}
\end{lem}
\begin{proof}
Here the calculations are motivated by Hoff \cite{Hoff1}. Taking the
material derivative to (\ref{3.2}), one deduces that
\begin{eqnarray}\label{3.10}
\nonumber&&\rho \ddot{u}+\nabla P_t+\textrm{div}(\nabla P\otimes
u)\\
&&=\mu[\Delta u_t+\textrm{div}(\Delta u\otimes
u)]+(\lambda+\mu)[\nabla\textrm{div}u_t+\textrm{div}((\nabla\textrm{div}u)\otimes
u)],
\end{eqnarray}
where $\textrm{div}(f\otimes u):=\sum\partial_j(fu^j)$ as in
\cite{Sun2,Sun3}. Taking the $L^2$ inner product of the above
equation with $\dot{u}$, and using $(\ref{1.1})_1$, we arrive at
\begin{eqnarray}\label{3.11}
\nonumber&&\frac{1}{2}\frac{d}{dt}\int\rho |\dot{u}|^2=\int
P_t\textrm{div}\dot{u}-\int\textrm{div}(\nabla P\otimes
u)\dot{u}\\
&&+\int\mu[\Delta u_t+\textrm{div}(\Delta u\otimes
u)]\dot{u}+(\lambda+\mu)\int[\nabla\textrm{div}u_t+\textrm{div}((\nabla\textrm{div}u)\otimes
u)]\dot{u},
\end{eqnarray}
A short computation shows that
\begin{equation}\label{3.12}
\int P_t\textrm{div}\dot{u}-\int\textrm{div}(\nabla P\otimes
u)\dot{u}=\int[P_t+\textrm{div}(Pu)]\textrm{div}\dot{u}-\int P\nabla
u:\nabla\dot{u}^{\top},
\end{equation}
\begin{eqnarray}\label{3.13}
\nonumber&&\int\mu[\Delta u_t+\textrm{div}(\Delta u\otimes
u)]\dot{u}\\
&=&-\mu\int|\nabla\dot{u}|^2+\mu\int[(\nabla u\nabla
u):\nabla\dot{u}+\nabla u:(\nabla
u\nabla\dot{u})-\textrm{div}u\nabla u:\nabla\dot{u}],
\end{eqnarray}
and
\begin{eqnarray}\label{3.14}
\nonumber&&(\lambda+\mu)\int[\nabla\textrm{div}u_t+\textrm{div}((\nabla\textrm{div}u)\otimes
u)]\dot{u}\\
&\nonumber=&-(\lambda+\mu)\int|\textrm{div}\dot{u}|^2+(\lambda+\mu)\int\textrm{div}
u\nabla u:\nabla\dot{u}^{\top}\\
&&+(\lambda+\mu)\int[\nabla u:\nabla
u^{\top}\textrm{div}\dot{u}-(\textrm{div}u)^2\textrm{div}\dot{u}],
\end{eqnarray}
Combining (\ref{3.11})-(\ref{3.14}), and using (\ref{3.7}) and
(\ref{3.00}) once more, we are led to
\begin{eqnarray}\label{3.15}
\nonumber&&\frac{1}{2}\frac{d}{dt}\int\rho
|\dot{u}|^2+\mu\int|\nabla\dot{u}|^2+(\lambda+\mu)\int|\textrm{div}\dot{u}|^2\\
&\nonumber=&R\int\rho\dot{\theta}\textrm{div}\dot{u}-\int P\nabla
u:\nabla\dot{u}^{\top}\\
&&\nonumber+\mu\int[(\nabla u\nabla u):\nabla\dot{u}+\nabla
u:(\nabla u\nabla\dot{u})-\textrm{div}u\nabla u:\nabla\dot{u}]\\
&&\nonumber+(\lambda+\mu)\int[\textrm{div} u\nabla
u:\nabla\dot{u}^{\top}+\nabla u:\nabla
u^{\top}\textrm{div}\dot{u}-(\textrm{div}u)^2\textrm{div}\dot{u}]\\
&\leq&\epsilon\int|\nabla\dot{u}|^2+C_\epsilon\int\sqrt{\rho}\dot{\theta}^2+C_\epsilon\int|\nabla
u|^2+C_\epsilon\int|\nabla u|^4.
\end{eqnarray}
Choosing a $\epsilon$ sufficiently small, the first term on the
right hand side of (\ref{3.15}) can be absorbed by the the left hand
side, and then integrating the resulting equation over $[0, t]$,
using (\ref{3.0}), we complete the proof of Lemma \ref{lem3.2}.
\end{proof}

\begin{lem}[\textbf{Crucial estimates}]\label{lem3.3}Under assumption (\ref{3.00}),  there holds for any $T<T^*$
\begin{equation}\label{3.16}
\sup_{0\leq t\leq T}\int(|\nabla
u|^2+\rho|\dot{u}|^2+|\nabla\theta|^2)+\int_0^T\int|\nabla\dot{u}|^2+\int_0^T\int\rho\dot{\theta}^2\leq
C.
\end{equation}
\end{lem}
\begin{proof}
According to $(\ref{1.1})_1$, the thermal energy equation
$(\ref{1.1})_3$ can be rewritten as
\begin{equation}\label{3.17}
c_\nu\rho\dot{\theta}-\kappa\Delta\theta+R\rho\theta\textrm{div}u=\frac{\mu}{2}|\nabla
u+\nabla u^{\top}|^2+\lambda(\textrm{div}u)^2.
\end{equation}
Multiplying the above equation by $\theta_t$, and integrating the
resulting equation over $\Omega$, we obtain
\begin{eqnarray}\label{3.18}
&&\nonumber
c_\nu\int\rho\dot{\theta}^2+\frac{\kappa}{2}\frac{d}{dt}\int|\nabla\theta|^2\\
&=&\nonumber
c_\nu\int\rho\dot{\theta}u\cdot\nabla\theta-R\int\rho\theta\textrm{div}u\dot{\theta}
+R\int\rho\theta\textrm{div}u
u\cdot\nabla\theta\\
&&\nonumber+\frac{d}{dt}\int[\frac{\mu}{2}|\nabla u+\nabla
u^{\top}|^2+\lambda(\textrm{div}u)^2]\theta\\
&&-\int[\mu(\nabla u+\nabla u^{\top}):(\nabla u_t+\nabla
u^{\top}_t)+2\lambda\textrm{div}u\textrm{div}u_t]\theta.
\end{eqnarray}
Note that
\begin{eqnarray}\label{3.19}
&&\nonumber-2\lambda\int\textrm{div}u\textrm{div}u_t\theta\\
&=&\nonumber-2\lambda\int\textrm{div}u\textrm{div}\dot{u}\theta+2\lambda\int\textrm{div}u\textrm{div}(u\cdot\nabla
u)\theta\\
&=&\nonumber -2\lambda\int\textrm{div}u\textrm{div}\dot{u}\theta+2\lambda\int\textrm{div}u\partial_i u^k\partial_k u^i\theta+2\lambda\int\textrm{div}u u^k\partial_k \textrm{div}u \theta \\
&=&\nonumber
-2\lambda\int\textrm{div}u\textrm{div}\dot{u}\theta+2\lambda\int\textrm{div}u\partial_i
u^k\partial_k u^i\theta+\lambda\int\theta
u\cdot\nabla(\textrm{div}u)^2\\
&=&\nonumber-2\lambda\int\textrm{div}u\textrm{div}\dot{u}\theta+2\lambda\int\textrm{div}u\nabla
u:\nabla u^{\top}\theta-\lambda\int\theta
(\textrm{div}u)^3-\lambda\int
u\cdot\nabla\theta (\textrm{div}u)^2,\\
\end{eqnarray}
and
\begin{eqnarray}\label{3.20}
-\mu\int(\nabla u+\nabla u^{\top}):(\nabla u_t+\nabla
u^{\top}_t)\theta&=&\nonumber-2\mu\int\nabla
u:\nabla u_t\theta-2\mu\int\nabla u:\nabla u_t^{\top}\theta\\
&=&I_1+I_2,
\end{eqnarray}
\begin{eqnarray}\label{3.21}
I_1&=&\nonumber-2\mu\int\nabla
u:\nabla\dot{u}\theta+2\mu\int\partial_i u^j\partial_i(u^k\partial_k
u^j)\theta\\
&=&\nonumber-2\mu\int\nabla u:\nabla\dot{u}\theta+2\mu\int\partial_i
u^j\partial_iu^k\partial_k u^j\theta+2\mu\int\partial_i
u^ju^k\partial_k \partial_iu^j\theta\\
&=&\nonumber-2\mu\int\nabla u:\nabla\dot{u}\theta+2\mu\int\nabla
u:(\nabla u\nabla u)\theta+\mu\int\theta u\cdot\nabla|\nabla u|^2\\
&=&\nonumber-2\mu\int\nabla u:\nabla\dot{u}\theta+2\mu\int\nabla
u:(\nabla u\nabla u)\theta-\mu\int\theta \textrm{div}u|\nabla
u|^2-\mu\int u\cdot\nabla\theta|\nabla u|^2,
\end{eqnarray}
\begin{eqnarray}\label{3.22}
I_2&=&\nonumber-2\mu\int\nabla
u:\nabla\dot{u}^{\top}\theta+2\mu\int\partial_i
u^j\partial_j(u^k\partial_k
u^i)\theta\\
&=&\nonumber-2\mu\int\nabla
u:\nabla\dot{u}^{\top}\theta+2\mu\int\partial_i
u^j\partial_ju^k\partial_k u^i\theta\\
&&\nonumber+\mu\int\partial_i u^ju^k\partial_k
\partial_ju^i\theta+\mu\int\partial_j
u^iu^k\partial_k \partial_iu^j\theta\\
&=&\nonumber-2\mu\int\nabla
u:\nabla\dot{u}^{\top}\theta+2\mu\int\nabla
u:(\nabla u\nabla u)^{\top}\theta+\mu\int\theta u^k\partial_k(\partial_iu^j\partial_ju^i)\\
&=&\nonumber-2\mu\int\nabla
u:\nabla\dot{u}^{\top}\theta+2\mu\int\nabla
u:(\nabla u\nabla u)^{\top}\theta\\
&&-\mu\int\theta \textrm{div}u\nabla u:\nabla u^{\top}-\mu\int
u\cdot\nabla\theta\nabla u:\nabla u^{\top}.
\end{eqnarray}
Substituting (\ref{3.19})-(\ref{3.22}) into (\ref{3.18}),  using
Cauchy's inequality and (\ref{3.0}), we get
\begin{eqnarray}\label{3.23}
&&\nonumber
c_\nu\int\rho\dot{\theta}^2+\frac{\kappa}{2}\frac{d}{dt}\int|\nabla\theta|^2\\
&\leq&\nonumber\frac{c_{\nu}}{2}\int\rho\dot{\theta}^2+C_\epsilon\int
|\nabla u|^2+C\int |u|^2|\nabla \theta|^2+\epsilon\int|\nabla\dot
u|^2\\
&&+C\int|\nabla u|^4+\frac{d}{dt}\int[\frac{\mu}{2}|\nabla u+\nabla
u^{\top}|^2+\lambda(\textrm{div}u)^2]\theta.
\end{eqnarray}
Consequently, integrating (\ref{3.23}) over $[0,t]$ and using
(\ref{3.0}) again, we are led to
\begin{eqnarray}\label{3.24}
\nonumber \int_0^t\int\rho\dot{\theta}^2+\int|\nabla\theta|^2 &\leq&
C_\epsilon+C\int_0^t\int |u|^2|\nabla
\theta|^2+\epsilon\int_0^t\int|\nabla\dot
u|^2\\
&&\nonumber+C\int_0^t\int|\nabla u|^4+C\int|\nabla u|^2.
\end{eqnarray}
Substituting (\ref{3.1}) and (\ref{3.9}) into the above inequality,
we have
\begin{eqnarray}\label{3.25}
&&\nonumber
\int_0^t\int\rho\dot{\theta}^2+\int|\nabla\theta|^2\\
&\leq&\nonumber C_\epsilon+C\int_0^t\int |u|^2|\nabla
\theta|^2+\epsilon C\int_0^t\int\rho|\dot{\theta}|^2\\
&&\nonumber+C\int_0^t\int\sqrt{\rho}|\dot{\theta}||\nabla
u|+C\int_0^t\int|\nabla u|^4\\
&\leq&\nonumber C_\epsilon+C\int_0^t\int |u|^2|\nabla
\theta|^2+\epsilon
C\int_0^t\int\rho|\dot{\theta}|^2+C\int_0^t\int|\nabla u|^4,
\end{eqnarray}
where we have used Cauchy's inequality and (\ref{3.0}). Take a
$\epsilon$ sufficiently small, then there holds
\begin{eqnarray}\label{3.26}
\int_0^t\int\rho\dot{\theta}^2+\int|\nabla\theta|^2 &\leq&
C+C\int_0^t\int |u|^2|\nabla \theta|^2+C\int_0^t\int|\nabla u|^4.
\end{eqnarray}
It follows from Lemma \ref{lem2.0}, Gagliardo-Nirenberg' inequality
and (\ref{3.00})  that
\begin{eqnarray}\label{3.27}
\int_0^t\int|\nabla u|^4&\leq&\nonumber \int_0^t\int|\nabla
v|^4+\int_0^t\int|\nabla w|^4\\
&\leq&\nonumber C\int_0^t\int|P|^4+C\int_0^t\|\nabla
w\|_{L^2}^2\|\nabla w\|_{H^1}^2\\
&\leq&\nonumber C+C\int_0^t(\|\nabla u\|_{L^2}^2+\|\nabla v\|_{L^2}^2)\|\sqrt{\rho}\dot{u}\|_{L^2}^2\\
&\leq&\nonumber C+C\int_0^t(\|P\|_{L^2}^2+\|\nabla
u\|_{L^2}^2)\|\sqrt{\rho}\dot{u}\|_{L^2}^2\\
&\leq&C+C\int_0^t(1+\|\nabla
u\|_{L^2}^2)\|\sqrt{\rho}\dot{u}\|_{L^2}^2,
\end{eqnarray}
In particular,
\begin{eqnarray}\label{3.28}
\|\nabla u\|_{L^4}^2&\leq&\nonumber
C\|P\|_{L^4}^2+C(\|P\|_{L^2}+\|\nabla
u\|_{L^2})\|\sqrt{\rho}\dot{u}\|_{L^2}\\
&\leq& C+C(1+\|\nabla u\|_{L^2})\|\sqrt{\rho}\dot{u}\|_{L^2}.
\end{eqnarray}
Moreover,
\begin{eqnarray}\label{3.29}
\int_0^t\int |u|^2|\nabla \theta|^2&\leq&
\int_0^t\|u\|_{L^\infty}^2\|\nabla \theta\|_{L^2}^2 \leq
C\int_0^t\|u\|_{L^\infty}^4+C\int_0^t\|\nabla \theta\|_{L^2}^4.
\end{eqnarray}
By Sobolev imbedding $W^{1,4}\hookrightarrow L^\infty$,
\begin{eqnarray}\label{3.30}
\|u\|_{L^\infty}^4&\leq C\|u\|_{L^4}^4+C\|\nabla u\|_{L^4}^4 \leq
C\|\nabla u\|_{L^4}^4,
\end{eqnarray}
where we have used Poincar\'{e}'s inequality. Then (\ref{3.29}) and
(\ref{3.30}) imply that
\begin{eqnarray}\label{3.31}
\int_0^t\int |u|^2|\nabla \theta|^2&\leq& C\int_0^t(\|\nabla
u\|_{L^4}^4+\|\nabla \theta\|_{L^2}^4).
\end{eqnarray}
Substituting (\ref{3.27}) and (\ref{3.31}) into (\ref{3.26}), we
obtain
\begin{eqnarray}\label{3.32}
&&\nonumber
\int_0^t\int\rho\dot{\theta}^2+\int|\nabla\theta|^2\\
&\leq& C+C\int_0^t\|\nabla \theta\|_{L^2}^4+C\int_0^t(1+\|\nabla
u\|_{L^2}^2)\|\sqrt{\rho}\dot{u}\|_{L^2}^2.
\end{eqnarray}
Now we are in a position to close all the estimates above. Indeed,
$4C\times$(\ref{3.32})+(\ref{3.1})+(\ref{3.9}) implies
\begin{eqnarray}\label{3.33}
&&\nonumber
\int(|\nabla u|^2+\rho|\dot{u}|^2+|\nabla\theta|^2)+\int_0^t\int|\nabla\dot{u}|^2+\int_0^t\int\rho\dot{\theta}^2\\
&\leq&\nonumber C+C\int_0^t\|\nabla
\theta\|_{L^2}^4+C\int_0^t(1+\|\nabla
u\|_{L^2}^2)\|\sqrt{\rho}\dot{u}\|_{L^2}^2\\
&\leq& C+C\int_0^t(1+\|\nabla u\|_{L^2}^2+\|\nabla
\theta\|_{L^2}^2)(\|\nabla
u\|_{L^2}^2+\|\sqrt{\rho}\dot{u}\|_{L^2}^2+\|\nabla
\theta\|_{L^2}^2).
\end{eqnarray}
By virtue of Gronwall's inequality and (\ref{3.0}), we complete the
proof of Lemma \ref{lem3.3}.
\end{proof}
\begin{cor}\label{cor3.4}Under assumption (\ref{3.00}),  there holds for any $T<T^*$
\begin{equation}\label{3.34}
\sup_{0\leq t\leq T}(\|u\|_{L^\infty}+\|\nabla u\|_{L^4})\leq C,
\end{equation}
\begin{equation}\label{3.34+}
\int_0^T\int|\nabla^2\theta|^2\leq C.
\end{equation}
\end{cor}
\begin{proof}
Clearly,  (\ref{3.34}) is a direct consequence of (\ref{3.16}),
(\ref{3.28}) and (\ref{3.30}). In view of (\ref{3.16}),
 (\ref{3.34}) and (\ref{3.00}), applying the standard elliptic regularity theory to (\ref{3.17}), then (\ref{3.34+}) follows
 immediately.
\end{proof}
\begin{lem}\label{lem3.5}Under assumption (\ref{3.00}),  there holds for any $T<T^*$
\begin{equation}\label{3.35}
\sup_{0\leq t\leq
T}\int\rho\dot{\theta}^2+\int_0^T\int|\nabla\dot{\theta}|^2\leq C.
\end{equation}
\begin{equation}\label{3.35+}
\sup_{0\leq t\leq T}\|\theta\|_{H^2}\leq C.
\end{equation}
\end{lem}
\begin{proof}
Taking the material derivative to (\ref{3.17}), we find
\begin{eqnarray}\label{3.36}
&&\nonumber
c_\nu\rho\ddot{\theta}-\kappa(\Delta\theta)\dot{}=(c_{\nu}-R)\rho\dot{\theta}\textrm{div}u+R\rho\theta[(\textrm{div}u)^2+\nabla
u:\nabla
u^{\top}]\\[2mm]
&-&\nonumber R\rho\theta\textrm{div}\dot{u}+[\mu(\nabla u+\nabla
u^{\top}):(\nabla \dot{u}+\nabla
\dot{u}^T)+2\lambda\textrm{div}u\textrm{div}\dot{u}]\\[2mm]
&-&[\mu(\nabla u+\nabla u^{\top}):(\nabla u\nabla u+\nabla
u^{\top}\nabla u^{\top})+2\lambda\textrm{div}u\nabla u:\nabla
u^{\top}],
\end{eqnarray}
where we have used the facts
\begin{equation*}
\dot{\rho}=-\rho\textrm{div}u, \ \ \
(\textrm{div}u)\dot{}=\textrm{div}\dot{u}-\nabla u:\nabla u^{\top}
\end{equation*}
and
\begin{equation*}
(\nabla u)\dot{}=\nabla \dot{u}-\nabla u\nabla u, \ \ \ (\nabla
u^{\top})\dot{}=\nabla \dot{u}^{\top}-\nabla u^{\top}\nabla
u^{\top}.
\end{equation*}
Taking the $L^2$ inner product of the (\ref{3.36}) with
$\dot{\theta}$, and using $(\ref{1.1})_1$, we arrive at
\begin{eqnarray}\label{3.37}
&&\nonumber
\frac{c_\nu}{2}\frac{d}{dt}\int\rho\dot{\theta}^2-\kappa\int(\Delta\theta)\dot{}\dot{\theta}\\
&\leq&\nonumber(c_{\nu}-R)\int\rho\dot{\theta}^2\textrm{div}u+R\int\rho\dot{\theta}\theta[(\textrm{div}u)^2+\nabla u:\nabla u^{\top}]\\
&& -R\int\rho\dot{\theta}\theta\textrm{div}\dot{u}+C\int|\nabla
u||\nabla \dot{u}||\dot{\theta}|+C\int|\nabla u|^3|\dot{\theta}|.
\end{eqnarray}
Note that
\begin{eqnarray}\label{3.38}
&&\nonumber-\kappa\int(\Delta\theta)\dot{}\dot{\theta}=-\kappa\int(\Delta\theta_t+u\cdot\nabla\Delta\theta)\dot{\theta}\\
&=&\nonumber\kappa\int\nabla \theta_t\nabla \dot{\theta}+\kappa\int\partial_i u^k\partial_{ki}\theta\dot{\theta}+\kappa\int u^k\partial_{ki}\theta\partial_i\dot{\theta}\\
&=&\nonumber\kappa\int |\nabla \dot{\theta}|^2-\kappa\int\nabla(u\cdot\nabla\theta)\cdot\nabla\dot{\theta}+\kappa\int\partial_i u^k\partial_{ki}\theta\dot{\theta}+\kappa\int u^k\partial_{ki}\theta\partial_i\dot{\theta}\\
&=&\kappa\int |\nabla
\dot{\theta}|^2-\kappa\int\nabla\dot{\theta}\nabla
u\nabla\theta+\kappa\int\nabla u:\nabla^2\theta\dot{\theta}.
\end{eqnarray}
We infer from (\ref{3.37}), (\ref{3.38}) and (\ref{3.00}) that
\begin{eqnarray}\label{3.39}
\nonumber
&&\frac{c_\nu}{2}\frac{d}{dt}\int\rho\dot{\theta}^2+\kappa\int
|\nabla
\dot{\theta}|^2\\
&\leq&\nonumber C\int\rho\dot{\theta}^2|\nabla
u|+C\int\rho|\dot{\theta}||\nabla
u|^2+C\int\rho|\dot{\theta}||\nabla\dot{u}|+C\int|\nabla u||\nabla
\dot{u}||\dot{\theta}|\\
&&\nonumber +C\int|\nabla
u|^3|\dot{\theta}|+\kappa\int|\nabla\dot{\theta}||\nabla
u||\nabla\theta|+\kappa\int|\nabla
u||\nabla^2\theta||\dot{\theta}|\\
&\leq&\nonumber\frac{\kappa}{2}\int|\nabla\dot{\theta}|^2+\epsilon\int|\nabla
u|^2|\dot{\theta}|^2+C_{\epsilon}\int\rho\dot{\theta}^2+C_\epsilon\int|\nabla
u|^4\\
&&+C_\epsilon\int|\nabla
\dot{u}|^2+C_{\epsilon}\int|\nabla^2\theta|^2+C\int|\nabla
u|^2|\nabla \theta|^2,
\end{eqnarray}
From (\ref{3.34}), Gagliardo-Nirenberg' inequality and
Poincar\'{e}'s inequalty, we have
\begin{eqnarray}\label{3.40}
\epsilon\int|\nabla u|^2|\dot{\theta}|^2&\leq&\epsilon\|\nabla
u\|_{L^4}^2\|\dot{\theta}\|_{L^4}^2\leq \epsilon
C\|\dot{\theta}\|_{L^2}\|\nabla\dot{\theta}\|_{L^2}\leq
C\|\nabla\dot{\theta}\|_{L^2}^2,
\end{eqnarray}
and
\begin{eqnarray}\label{3.41}
\nonumber\int|\nabla u|^2|\nabla\theta|^2&\leq&\|\nabla
u\|_{L^4}^2\|\nabla\theta\|_{L^4}^2\leq
C\|\nabla\theta\|_{L^2}\|\nabla\theta\|_{H^1}\\
&\leq&C\|\nabla\theta\|_{L^2}^2+C\|\nabla^2\theta\|_{L^2}^2\leq
C+C\|\nabla^2\theta\|_{L^2}^2,
\end{eqnarray}
where we have used (\ref{3.16}) in the last inequality above.
Substituting (\ref{3.40}) and (\ref{3.41}) into (\ref{3.39}),
choosing a $\epsilon$ small enough and using (\ref{3.34}), we obtian
\begin{eqnarray}\label{3.42}
\nonumber &&c_\nu\frac{d}{dt}\int\rho\dot{\theta}^2+\kappa\int
|\nabla \dot{\theta}|^2 \leq C+C\int|\nabla^2\theta|^2+C\int|\nabla
\dot{u}|^2+C\int\rho\dot{\theta}^2.
\end{eqnarray}
Integrating the above equation over $[0,t]$, and using (\ref{3.16})
and (\ref{3.34+}) once more, we get (\ref{3.35}). Finally,
(\ref{3.35+}) follows from (\ref{3.17}), (\ref{3.16}), (\ref{3.34}),
(\ref{3.35})
 and (\ref{3.00}) immediately. This completes the proof of Lemma \ref{lem3.5}.
\end{proof}
\section{Improved regularity of the density and velocity}
Generally speaking, for the continuity equation $(\ref{1.1})_1$, the
gradient estimate on $\rho$ rely on the boundedness of
$\int_0^T\|\nabla u\|_{L^\infty}$. On the other hand, due to the
lower regularity on $\rho$, we can not obtain the higher order
regularity of $u$ through the momentum equation (\ref{3.2}).
Nevertheless, with the help of the decomposition $u=v+w$, we can
close the gradient estimate of $\rho$ based on a logarithmic
estimate for the strongly elliptic operator
$\mu\Delta+(\lambda+\mu)\nabla\textrm{div}$.

We first give some higher regularity on $w$ below.
\begin{lem}\label{lem4.1}
Under assumption (\ref{3.00}),  there holds for any $T<T^*$
\begin{equation}\label{4.4}
\sup_{0\leq t\leq T}\|w\|_{H^2}+\int_0^T(\|\nabla^2
w\|_{L^q}^2+\|\nabla w\|_{L^\infty}^2)\leq C,\ \ q\in(2,\infty).
\end{equation}
\end{lem}
\begin{proof}
Using (\ref{2.01}) with $U=w, F=\rho\dot{u}$ and $p=2$ , we have
\begin{equation}\label{4.5}
\| w\|_{H^2}\leq C \|\rho\dot{u}\|_{L^2}\leq
C\|\sqrt{\rho}\dot{u}\|_{L^2},
\end{equation}
where we have used (\ref{3.00}). Then it follows from (\ref{3.16})
that
\begin{equation}\label{4.7}
\sup_{0\leq t\leq T}\|w\|_{H^2}\leq C.
\end{equation}
For $q>2$, using (\ref{2.01}), (\ref{3.00}) again, and by virtue of
Gagliardo-Nirenberg inequality, we find
\begin{eqnarray}\label{4.8}
\nonumber&&\|\nabla^2 w\|_{L^q}\leq C \|\rho\dot{u}\|_{L^q}\leq
C\|\dot{u}\|_{L^q}\\[2mm]
\nonumber
&\leq&C\|\dot{u}\|_{L^2}^{\frac{2}{q}}\|\nabla\dot{u}\|_{L^2}^{\frac{q-2}{q}}\leq
C\|\nabla\dot{u}\|_{L^2},
\end{eqnarray}
where we have used Poincar\'{e}'s inequality. Next, we use the
Sobolev imbeding $W^{1,q}\hookrightarrow L^\infty$ for $q>2$ to get
\begin{eqnarray}\label{4.9}
\|\nabla w\|_{L^\infty}\leq C (\|\nabla w\|_{L^q}+\|\nabla^2
w\|_{L^q})\leq C(\|\nabla w\|_{H^1}+\|\nabla^2 w\|_{L^q}),
\end{eqnarray}
then we infer from (\ref{3.16}), (\ref{4.7}) and (\ref{4.9}) that
\begin{equation}\label{4.10}
\int_0^T(\|\nabla^2 w\|_{L^q}^2+\|\nabla w\|_{L^\infty}^2)\leq C.
\end{equation}
This completes the proof of Lemma \ref{lem4.1}.
\end{proof}
Thanks to all the estimates obtained above, we will get the gradient
estimates of the density $\rho$ next.
\begin{lem}\label{lem4.2}Under assumption (\ref{3.00}),  there holds for any $T<T^*$
\begin{equation}\label{4.11}
\sup_{0\leq t\leq T}(\|\rho\|_{W^{1,
q}}+\|u\|_{H^2})+\int_0^T\|\nabla u\|_{L^\infty}^2\leq C.
\end{equation}
\end{lem}
\begin{proof}
The proof follows the ideas of Sun, Wang and Zhang
\cite{Sun1,Sun2,Sun3}, we sketch it here for completeness. First of
all, take $\nabla$ to the continuity equation $(\ref{1.1})_1$ to
find
\begin{equation}\label{4.12}
\partial_t
\nabla\rho+(u\cdot\nabla)\nabla\rho+\nabla
u\nabla\rho+\textrm{div}u\nabla\rho+\rho\nabla\textrm{div}u=0.
\end{equation}
Multiplying (\ref{4.12}) by $q|\nabla\rho|^{q-2}\nabla\rho$  and
integrating over $\Omega$, we obtain
\begin{eqnarray}\label{4.13}
\nonumber\frac{d}{dt}\|\nabla\rho\|_{L^q}&\leq& C\|\nabla
u\|_{L^\infty}\|\nabla \rho\|_{L^q}+C\|\nabla^2u\|_{L^q}\\[2mm]
\nonumber&\leq& C(\|\nabla v\|_{L^\infty}+\|\nabla
w\|_{L^\infty})\|\nabla
\rho\|_{L^q}+C(\|\nabla^2v\|_{L^q}+\|\nabla^2w\|_{L^q})\\[2mm]
\nonumber&\leq& C(1+\|\nabla v\|_{L^\infty}+\|\nabla
w\|_{L^\infty})\|\nabla
\rho\|_{L^q}+C(\|\nabla^2w\|_{L^q}+\|\nabla\theta\|_{L^q}),\\
\end{eqnarray}
where we have used (\ref{2.01}). To close (\ref{4.13}), we have to
bound $\|\nabla v\|_{L^\infty}$, and it is just this term leads us
to show the endpoint estimate for the strongly elliptic operator
$\mu\Delta+(\lambda+\mu)\nabla\textrm{div}$. In fact, Lemma
\ref{lem2.0}-\ref{lem2.4} imply that if $q>2$
\begin{eqnarray}\label{4.14}
\nonumber\|\nabla v\|_{L^\infty}&\leq& C(1+\|\nabla
v\|_{BMO}\ln(e+\|\nabla^2 v\|_{L^q}))\\[2mm]
\nonumber&\leq&C(1+(\|P\|_{L^\infty}+\|P\|_{L^2})\ln(e+\|\nabla^2
v\|_{L^q}))\\[2mm]
&\nonumber\leq&C(1+\ln(e+\|\nabla \rho\|_{L^q})+\ln(e+\|\nabla
\theta\|_{L^q}))\\[2mm]
&\leq&C(1+\|\nabla \theta\|_{L^q}+\ln(e+\|\nabla \rho\|_{L^q})).
\end{eqnarray}
Substituting (\ref{4.14}) into (\ref{4.13}), we get
\begin{eqnarray}\label{4.15}
\nonumber\frac{d}{dt}(e+\|\nabla\rho\|_{L^q})&\leq& C(1+\|\nabla
\theta\|_{L^q}+\|\nabla w\|_{L^\infty})\|\nabla
\rho\|_{L^q}+C\ln(e+\|\nabla \rho\|_{L^q})\|\nabla
\rho\|_{L^q}\\[2mm]
&& +C(\|\nabla^2w\|_{L^q}+\|\nabla\theta\|_{L^q}),
\end{eqnarray}
then using (\ref{4.4}), (\ref{3.35+}) and Gronwall's inequality, one
deduces that
\begin{equation}\label{4.16}
\sup_{0\leq t\leq T}\|\nabla \rho\|_{L^q}\leq C.
\end{equation}
From (\ref{3.35+}), (\ref{4.4}), (\ref{4.16}) and Lemma
\ref{lem2.0}, there holds for $q>2$
\begin{eqnarray}\label{4.17}
\nonumber\int_0^T\|\nabla u\|_{L^\infty}^2&\leq& C\int_0^T\|\nabla
v\|_{L^\infty}^2+\int_0^T\|\nabla w\|_{L^\infty}^2\\
&\leq& C+C\int_0^T\|\nabla v\|_{L^q}^2+C\int_0^T\|\nabla^2
v\|_{L^q}^2\leq C.
\end{eqnarray}
Moreover, it follows from (\ref{4.4}), (\ref{3.35+}) and
(\ref{4.16})
\begin{equation}\label{4.18}
\|\nabla^2 u\|_{L^2}\leq \|\nabla^2 w\|_{L^2}+\|\nabla^2
v\|_{L^2}\leq C+C\|\nabla \rho\|_{L^2}\leq C.
\end{equation}
This completes the proof of Lemma \ref{lem4.2}.
\end{proof}
\section{Proof of Theorem \ref{thmBC}}
The combination of Lemma \ref{lem3.3}, Lemma \ref{lem3.5} and Lemma
\ref{lem4.2} will enable us to extend the strong solution $(\rho, u,
\theta)$ beyond the maximal existence time $T^*$. Indeed, by virtue
of (\ref{4.11}), (\ref{3.35+}) and time continuity stated in
(\ref{pro3}), we can define
\begin{equation}\label{5.1}
(\rho, u, \theta)|_{t=T^*}=\lim_{t\rightarrow T^*}(\rho, u, \theta),
\end{equation}
and
\begin{equation}\label{5.2}
h:=\rho\dot{u}|_{t=T^*}=\lim_{t\rightarrow T^*}(\mu\Delta
u+(\lambda+\mu)\nabla\textrm{div}u-\nabla P)\quad \textrm{strongly
in}\  L^2.
\end{equation}
On the other hand, using Sobolev imbedding $W^{1, q}\hookrightarrow
C^{0,1-\frac{2}{q}}$ for $q>2$, (\ref{4.11}) and the time continuity
on $\rho$ stated in (\ref{pro3}), one easily deduces that
\begin{equation*}
\rho|_{t=T^*}(x)=\lim_{t\rightarrow T^*}\rho(x) \quad
\textrm{uniformly in}\ \Omega,
\end{equation*}
and hence
\begin{equation}\label{5.3}
\sqrt{\rho}|_{t=T^*}(x)=\lim_{t\rightarrow T^*}\sqrt{\rho}(x) \quad
\textrm{pointwise in}\ \Omega,
\end{equation}
then (\ref{5.3}), the assumption (\ref{3.00}) that $\rho$ is upper
bounded and Lebesgue's Dominated Convergence Theorem imply that
\begin{equation}\label{5.4}
\sqrt{\rho}|_{t=T^*}=\lim_{t\rightarrow T^*}\sqrt{\rho} \quad
\textrm{strongly in}\ L^2.
\end{equation}
Besides, we infer from $(\ref{3.16})$ that there exists a sequence
$t_k\rightarrow T^*$ as $k\rightarrow\infty$ and a function
$\tilde{g}_1\in L^2$, such that
\begin{equation}\label{5.5}
\tilde{g}_1=\lim_{k\rightarrow \infty}(\sqrt{\rho}\dot{u})(t_k)
\quad \textrm{weakly in}\ L^2.
\end{equation}
It follows from (\ref{5.4}) and (\ref{5.5}) that
\begin{equation}\label{5.6}
\sqrt{\rho}|_{t=T^*}\tilde{g}_1=\lim_{k\rightarrow
\infty}(\rho\dot{u})(t_k) \quad \textrm{ in\ the\ sence\ of\
distribution}.
\end{equation}
Comparing (\ref{5.2}) with (\ref{5.6}), we obtain
\begin{equation*}
h=\sqrt{\rho}|_{t=T^*}\tilde{g}_1,
\end{equation*}
i.e.
\begin{equation}\label{5.7}
(\mu\Delta u+(\lambda+\mu)\nabla\textrm{div}u-\nabla
P)|_{t=T^*}=\sqrt{\rho}|_{t=T^*}\tilde{g}_1.
\end{equation}
Similarly, there exists a function $\tilde{g}_2\in L^2$ such that
\begin{equation}\label{5.8}
(\kappa\Delta\theta+\frac{\mu}{2}|\nabla u+\nabla
u^{\top}|^2+\lambda(\textrm{div}u)^2)|_{t=T^*}=\sqrt{\rho}|_{t=T^*}\tilde{g}_2.
\end{equation}
Now (\ref{5.1}), (\ref{5.7}) and (\ref{5.8}) assure that we can take
$(\rho, u, \theta)|_{t=T^*}$ as the initial data and apply
Proposition \ref{prop1} to extend the local strong solution beyond
$T^*$, which contradicts the maximality of $T^*$. This completes the
proof of Theorem \ref{thmBC}. {\hfill $\square$\medskip}

\bigbreak\noindent{ \large\bf Acknowledgment.} The second author
would like to thank Professor Zhifei Zhang and Dr. Huanyao Wen for
their sincere help.

\end{document}